\documentclass[11pt]{article}
\usepackage{a4wide}

\usepackage[a4paper, margin=2.3cm]{geometry}
\usepackage{amsmath,amssymb,amsthm}
\usepackage{color}
\usepackage{parskip}
\usepackage{hyperref}
\usepackage{parskip}
\usepackage{setspace}
\usepackage{graphicx}
\usepackage{mathtools}
\usepackage[thinc]{esdiff}
 
\newtheorem{theorem}{Theorem}[section]

\newtheorem{corollary}[theorem]{Corollary}
\newtheorem{lemma}[theorem]{Lemma}
\newtheorem{question}[theorem]{Question}

\newtheorem*{definition*}{Definition}

\newcommand{\Comment}[1]{}

\theoremstyle{remark}

\begin{document}
\title{A point-sphere incidence bound in odd dimensions and applications}
\author{Doowon Koh\thanks{Department of Mathematiacs, Chungbuk National University, Korea. Email: koh131@chungbuk.ac.kr}\and Thang Pham\thanks{Department of Mathematics, HUS, Vietnam National University. Email: phamanhthang.vnu@gmail.com}}
\date{}
\maketitle

\begin{abstract}
In this paper, we prove a new point-sphere incidence bound in vector spaces over finite fields. More precisely, let $P$ be a set of points and $S$ be a set of spheres in $\mathbb{F}_q^d$. Suppose that $|P|, |S|\le N$, we prove that the number of incidences between $P$ and $S$ satisfies
\[I(P, S)\le N^2q^{-1}+q^{\frac{d-1}{2}}N,\]
under some conditions on $d, q$, and radii. This improves the known upper bound $N^2q^{-1}+q^{\frac{d}{2}}N$ in the literature. As an application, we show that for $A\subset \mathbb{F}_q$ with $q^{1/2}\ll |A|\ll q^{\frac{d^2+1}{2d^2}}$, one has
\[\max \left\lbrace |A+A|,~ |dA^2|\right\rbrace \gg \frac{|A|^d}{q^{\frac{d-1}{2}}}.\]
This improves earlier results on this sum-product type problem over arbitrary finite fields.
\end{abstract}

\section{Introduction}
Let $\mathbb{F}_q$ be a finite field of order $q$, where $q$ is a prime power. A sphere centered at $(c_1, \ldots, c_d)\in \mathbb{F}_q^d$ of radius $r$ is defined by the equation 
\[(x_1-c_1)^2+\cdots+(x_d-c_d)^2=r.\]
Let $P$ be a set of points in $\mathbb{F}_q^d$ and $S$ be a set of spheres with arbitrary radii in $\mathbb{F}_q^d$.
Let $I(P, S)$ be the number of incidences between $P$ and $S$, namely, 
\[I(P, S)=\#\left\lbrace (p, s)\colon p\in s, p\in P, s\in S\right\rbrace.\] 
The following point-sphere incidence bound was obtained  by Cilleruelo, Iosevich, Lund, Roche-Newton, and Rudnev  \cite{CILRR17}, and independently by Phuong, Pham, and Vinh \cite{PPV}.
\begin{theorem}\label{motbig}
Let $P$ be a set of points in $\mathbb{F}_q^d$ and $S$ be a set of spheres with arbitrary radii in $\mathbb{F}_q^d$. We have
\begin{equation}\label{square*}\left\vert I(P, S)-\frac{|P||S|}{q}\right\vert\le q^{\frac{d}{2}}(|P||S|)^{1/2}.\end{equation}
\end{theorem}

It follows from this theorem that if $|P||S|>q^{d+2}$, then the set of point-sphere incidences between $P$ and $S$ is non-empty. Using this incidence bound, Cilleruelo et al. \cite{CILRR17} proved a Beck type theorem that for any set $P\subset \mathbb{F}_q^2$, if $|P|>5q$, then the number of distinct circles determined by points of $P$ is at least $\frac{4q^3}{9}$. We refer the reader to \cite{CILRR17, PPV} for other applications. 

Throughout this paper, we use the following notations: $X \ll Y$ means that there exists  some absolute constant $C_1>0$ such that $X \leq C_1Y$, and  $X\sim Y$ means $Y\ll X\ll Y$.

The main purpose of this paper is to improve the upper bound of Theorem \ref{motbig}. More precisely, we study the following question: 

\begin{question}\label{qs}
Let $P$ be a set of points and $S$ be a set of spheres in $\mathbb{F}_q^d$. Under what conditions on $d$, $q$, and the sets will we be able to improve the bound $\frac{|P||S|}{q}+q^{\frac{d}{2}}\sqrt{|P||S|}$?
\end{question}
We first start with some observations. 

{\bf Observation $1$:} if $d\equiv 3\mod 4$ and $q\equiv 1\mod 4$, or $d\equiv 1\mod 4$, then the term $q^{\frac{d}{2}}(|P||S|)^{1/2}$ in (\ref{square*}) cannot be improved to $q^{\frac{d}{2}-\epsilon}(|P||S|)^{1/2}$ for any $\epsilon>0$ for arbitrary sets $P$ and $S$. Otherwise, one could follow the proof of \cite[Corollary 1]{CILRR17} to show that for any $E\subset \mathbb{F}_q^d$ with $|E|\gg q^{\frac{d+1}{2}-\epsilon'}$, for some $\epsilon'>0$, we have the set of distances determined by pairs of points in $E$ satisfies $|\Delta(E)|\gg q$. This would contradict with a construction in \cite[Theorem 2.7]{hart} that states that the exponent $\frac{d+1}{2}$ for the distance problem is sharp in those dimensions even one wishes to cover a positive proportion of all distances. Note that in the proof of \cite[Corollary 1]{CILRR17}, the size of $S$ is much larger than the number of points in $P$. We refer the reader to \cite{CILRR17, hart} for more explanations.

{\bf Observation $2$:} if $d\equiv 2\mod 4$ and $q\equiv 1\mod 4$, or $d\equiv 0\mod 4$, then there exists a set $E\subset \mathbb{F}_q^d$ with $|E|=q^{\frac{d}{2}}$ such that $x\cdot y=x\cdot x=0$ for all $x, y\in E$, see \cite[Lemma 5.1]{hart}. Hence, we can set $P=E$ and $S$ being the set of spheres centered at points in $E$ of radius $0$. It is clear that $I(P, S)=|P||S|=q^{\frac{d}{2}}\sqrt{|P||S|}$. Thus, the upper bound of (\ref{square*}) is sharp for this case. 

{\bf Observation $3$:} if all spheres in $S$ have the same radius, then a stronger result follows directly from a theorem of Iosevich and Rudnev in \cite{iosevich-rudnev}:
\begin{equation}\left\vert I(P, S)-\frac{|P||S|}{q}\right\vert< 2q^{\frac{d-1}{2}}(|P||S|)^{1/2}.\end{equation}

In a recent paper \cite{KLP}, Koh, Pham, and Lee introduced an approach of using results from the restriction problem for cones to study this incidence topic. As a consequence, they obtained the following improvement.
\begin{theorem} \label{incidence-theorem}
Let $P$ be a set of points in $\mathbb{F}_q^d$ and $S$ be a set of spheres in $ \mathbb{F}_q^d.$ 
\begin{enumerate}
\item If $d\equiv 2 \mod{4}$, $q\equiv 3 \mod{4}$, and $|S|\le q^{\frac{d}{2}}$, then we have
$$\left|I(P, S)-q^{-1}|P||S| \right|\ll q^{\frac{d-1}{2}}|P|^{\frac{1}{2}}|S|^{\frac{1}{2}}.$$
\item If  $d$ is even and $q\equiv 1 \mod{4},$ or $d\equiv 0 \mod{4},$   then the same conclusion holds under the condition $|S|\le q^{\frac{d-2}{2}}$. 
\item If $d\ge 3$ is odd, then the same conclusion holds under the condition $|S|\le q^{\frac{d-1}{2}}$. 
\end{enumerate}
\end{theorem}
In comparison, in its ranges, Theorem \ref{incidence-theorem} improves Theorem \ref{motbig} in both lower and upper bounds. Theorem \ref{incidence-theorem} is sharp in the sense that one can construct sets $P$ and $S$ with $|S|$ arbitrary small and $|P||S|\le q^{d+1}$ such that $I(P, S)=0$.

In the first result we provide an improvement in odd dimensions when $|P|\sim |S|$. 
\begin{theorem}\label{new-incidence}
Let $P$ be a set of points and $S$ be a set of spheres of \textbf{square} radii in $\mathbb{F}_q^d$. Suppose that $d\equiv 3\mod 4$ and $q\equiv 3\mod 4$. If $|P|, |S|\le N$, then we have 
\[I(P, S)\ll q^{-1}N^2+q^{\frac{d-1}{2}}N.\]
\end{theorem}
It is worth noting that one cannot expect to prove the same upper bound in even dimensions ($d\equiv 2\mod 4$ and $q\equiv 1\mod 4,$ or $d\equiv 0 \mod 4$). This follows from the second observation above.
\begin{corollary}
Let $P$ be a set of points and $S$ be a set of spheres of \textbf{square} radii in $\mathbb{F}_q^3$. Suppose that $q\equiv 3\mod 4$, and $|P|=|S|\sim q^2$, then we have 
\[I(P, S)\ll |P|^{3/4}|S|^{3/4}.\]
\end{corollary}
The proof of Theorem \ref{new-incidence} is based on a careful analysis of spectrum of graphs defined by cone equations. In particular, let $C_k$ be the cone in $\mathbb F_q^k$ defined by 
 \begin{equation}\label{moi}C_k:=\{x\in \mathbb{F}_q^k\colon Q(x)=-x_1^2+x_2^2+\cdots+x_k^2=0\}.\end{equation}
Let $G_{Q, k}$ be the Cayley graph with the vertex set $\mathbb{F}_q^k$, there is an edge between two vertices $x$ and $y$ if and only if $x-y\in C_k$.  It is clear that $G_{Q, k}$ is a regular graph of order $|C_k|$. In the following theorem, we show that when $k\equiv 0\mod 4$ and $q\equiv 3\mod 4$,  the unique positive and non-trivial eigenvalue of this graph is much smaller than the absolute value of others. This observation plays the main role in the proof of Theorem \ref{new-incidence}.
\begin{theorem}\label{eigenvalue}
Let $\{\lambda_{m}\}_{m\in \mathbb{F}_q^k}$ be the eigenvalues of $G_{Q, k}$. If $k\equiv 0\mod 4$ and $q\equiv 3 \mod{4},$ then we have
 $$ \lambda_m=q^k\cdot \left\{ \begin{array}{ll} q^{-1}\delta_0(m)- q^{-\frac{k}{2}}+q^{-\frac{(k+2)}{2}}  \quad &\mbox{if}~~ m\in C_k\\
 q^{-\frac{(k+2)}{2}}   \quad &\mbox{if} ~~ m\notin C_k. \end{array}\right.$$
Here, and throughout the paper, we define $\delta_0(m)=1$ if  $m=(0, \ldots, 0)$, and $\delta_0(m)=0$ otherwise.
\end{theorem}

Our next improvement is for spheres of non-square radii in dimensions $d\equiv 1\mod 4$. 
\begin{theorem}\label{new-incidence11}
Let $P$ be a set of points and $S$ be a set of spheres of \textbf{non-square} radii in $\mathbb{F}_q^d$. Suppose that $d\equiv 1\mod 4$ and $q\equiv 3\mod 4$. If $|P|, |S|\le N$, then we have 
\[I(P, S)\ll q^{-1}N^2+q^{\frac{d-1}{2}}N.\]
\end{theorem}
\begin{corollary}
Let $P$ be a set of points and $S$ be a set of spheres of \textbf{non-square} radii in $\mathbb{F}_q^5$. Suppose that $q\equiv 3\mod 4$, and $|P|=|S|\sim q^3$, then we have 
\[I(P, S)\ll |P|^{5/6}|S|^{5/6}.\]
\end{corollary}
Unlike Theorem \ref{new-incidence}, we do not have any construction to show that the upper bound $q^{-1}N^2+q^{\frac{d-1}{2}}N$ is impossible for even dimensions. 

Theorem \ref{new-incidence11} is proved by the same approach as for Theorem \ref{new-incidence}. The main difference is that we use the Cayley graph defined by the zero-norm equation. In particular, let $S_0^{k-1}$ be the sphere centered at the origin of radius zero in $\mathbb F_q^k$ defined by 
 \begin{equation}\label{moi11}S_0^{k-1}:=\{x\in \mathbb{F}_q^k\colon ||x||:=x_1^2+x_2^2+\cdots+x_k^2=0\}.\end{equation}
Let $G_{||\cdot ||, k}$ be the Cayley graph with the vertex set $\mathbb{F}_q^k$, there is an edge between two vertices $x$ and $y$ if and only if $x-y\in S_0^{k-1}$.  It is clear that $G_{||\cdot||, k}$ is a regular graph of order $|S_0^{k-1}|$. As in the graph $G_{Q, k}$, in the following theorem, we show that when $k\equiv 2\mod 4$ and $q\equiv 3\mod 4$,  the unique positive and non-trivial eigenvalue of this graph is much smaller than the absolute value of others. 
\begin{theorem}\label{eigenvalue198}
Let $\{\lambda_{m}\}_{m\in \mathbb{F}_q^k}$ be the eigenvalues of $G_{||\cdot||, k}$. If $k\equiv 2\mod 4$ and $q\equiv 3 \mod{4},$ then we have
 $$ \lambda_m=q^k\cdot \left\{ \begin{array}{ll} q^{-1}\delta_0(m)- q^{-\frac{k}{2}} + q^{-\frac{k+2}{2}} \quad &\mbox{if}~~ ||m||=0\\
 q^{-\frac{(k+2)}{2}}   \quad &\mbox{if} ~~ ||m||\ne 0. \end{array}\right.$$
\end{theorem}
In graph theoretic point of view, we believe that Theorems \ref{eigenvalue} and \ref{eigenvalue198} have a potential for applications to other topics.

\paragraph{Sharpness of Theorems \ref{new-incidence} and \ref{new-incidence11}:} Both Theorems \ref{new-incidence} and \ref{new-incidence11} cannot improved when $N>q^{\frac{d+1}{2}}$. The simplest example is to take $S$ being a set of spheres with the same radius, then Observation $3$ would tell us that $I(P, S)\sim N^2/q$. When $N>q^{\frac{d+2}{2}}$, Theorem \ref{motbig} also tells us that the number of incidences is at least $(1-o(1))N^2/q$.

We now provide some applications.

\paragraph{\textbf{Erd\H{o}s-Falconer distance problem}:} For any two points $x$ and $y$ in $\mathbb{F}_q^d$, we define its distance function by $||x-y||=(x_1-y_1)^2+\cdots+(x_d-y_d)^2$. For $E\subset \mathbb{F}_q^d$ and $t\ne 0$, let $U(t)$ be the number of pairs of points in $E$ of distance $t$. Iosevich and Rudnev \cite{iosevich-rudnev}, using the Kloosterman sum,  proved that 
\begin{equation}\label{upper}
\frac{|E|^2}{q}-2q^{\frac{d-1}{2}}|E|\le U(t)\le \frac{|E|^2}{q}+2q^{\frac{d-1}{2}}|E|
\end{equation}
As a consequence of Theorem \ref{new-incidence}, we can see that the upper bound of (\ref{upper}) can be recovered when $t$ is a square. The same holds when $t$ is a non-square by Theorem \ref{new-incidence11}. The most interesting aspect of this observation is that we are able to use Gauss sums instead of Kloosterman sum in the proof. It is still an open question whether or not one can prove the lower bound of (\ref{upper}) without the Kloosterman sum. 
\vspace*{0.5cm}
\paragraph{\textbf{A sum-product type estimate:}} 
For $A\subset \mathbb{F}_q$,  we define 
\[A+A:=\{a+b\colon a, b\in A\}, ~A^2:=\{a^2\colon a\in A\}, ~nA^2=\left\lbrace a_1+\cdots+a_n\colon a_i\in A^2\right\rbrace.\]
As a consequence of Theorem \ref{new-incidence}, we obtain the following sum-product type estimate. 


\begin{theorem}\label{sum-product}
Let $A$ be a set in $\mathbb{F}_q$ with $q\equiv 3\mod 4$ and $|A|\gg q^{1/2}$. For $d\ge 3$ odd, we  have at least one of two following statements:
\begin{enumerate}
\item $|A+A|\ge \min \left\lbrace q^{\frac{d+1}{2d}}, |A|^{\frac{d+1}{d}}\right\rbrace$.
\item $|dA^2|\gg \frac{|A|^d}{q^{\frac{d-1}{2}}}$. 
\end{enumerate}
\end{theorem}

\begin{corollary}\label{batap}
Let $A$ be a set in $\mathbb{F}_q$ with $q\equiv 3\mod 4$ and $q^{1/2}\ll |A|\ll q^{\frac{d^2+1}{2d^2}}$. For $d\ge 3$ odd, we have
\[\max \left\lbrace |A+A|,~ |dA^2|\right\rbrace \gg \frac{|A|^d}{q^{\frac{d-1}{2}}}.\]
In particular, for $d=3$, one has 
\[\max \{|A+A|, |A^2+A^2+A^2|\}\gg \frac{|A|^3}{q}.\]
\end{corollary}
The lower bound $|A|^dq^{-\frac{d-1}{2}}$ improves earlier results in the literature, for instance, $|A|^{\frac{3d-5}{d-1}}q^{\frac{2-d}{d-1}}$ in \cite{hiep}. We refer the reader to \cite{hiep} for discussions on this sum-product type problem, and to \cite{M, Rudnev} and references therein for results on other types.

The rest of this paper is organized as follows. In the next section, we recall some notations from  discrete Fourier analysis, and proofs of Theorems \ref{eigenvalue} and \ref{eigenvalue198} are given in Sections \ref{sec3} and \ref{sec4}, respectively. In Section \ref{sec5}, we provide proofs of Theorems \ref{new-incidence} and \ref{new-incidence11}. In Section \ref{sec6}, a proof of Theorem \ref{sum-product} is presented. In the last section, we address some open questions.


\section{Preliminaries}
We first recall some notations and lemmas from discrete Fourier analysis. Let $f$ be a complex valued function on $\mathbb F_q^k.$ The Fourier transform $\widehat{f}$ of $f$ is defined by
$$ \widehat{f}(m):=q^{-k} \sum_{x\in \mathbb F_q^k} \chi(-m\cdot x) f(x),$$ 
where $\chi$ denotes the principal additive character of $\mathbb F_q.$ The Fourier inversion theorem states that
$$ f(x)=\sum_{m\in \mathbb F_q^k} \chi(m\cdot x) \widehat{f}(m).$$ 
The orthogonality of the additive character $\chi$ says that
$$ \sum_{\alpha\in \mathbb F_q^k} \chi(\beta\cdot \alpha)
=\left\{\begin{array}{ll} 0\quad &\mbox{if}\quad \beta\ne (0,\ldots, 0),\\
q^k\quad &\mbox{if}\quad \beta=(0,\ldots,0). \end{array}\right.$$

As a direct application of the orthogonality of $\chi$, we obtain
$$ \sum_{m\in \mathbb F_q^k} |\widehat{f}(m)|^2 =q^{-k}\sum_{x\in \mathbb F_q^k} |f(x)|^2,$$
which is known as the Plancherel theorem. 

For example, it follows from the Plancherel theorem  that for any set $E$ in $\mathbb F_q^k,$ 
$$ \sum_{m\in \mathbb F_q^k} |\widehat{E}(m)|^2= q^{-k}|E|.$$
Here, and throughout this note,  we identify a set $E$ with the indicator function $1_E$ on $E.$

Throughout this paper,  let $\eta$ be the quadratic character of $\mathbb F_q$, namely, for $s\ne 0$, $\eta(s)=1$ if $s$ is a square, and $\eta(s)=-1$ if $s$ is a non-square. We also use the convention that $\eta(0)=0.$ 

For $a\in \mathbb F_q^*,$ the Gauss sum ${\mathcal G_a}$ is defined by
\begin{equation}\label{GaussDef}
{\mathcal G_a}:=\sum_{s\in \mathbb F_q^*}\eta(s) \chi(as),\end{equation}  which can be written as
$$ {\mathcal G_a}=\sum_{s\in \mathbb F_q} \chi(as^2)=\eta(a) {\mathcal G_1}.$$
The absolute value of the Gauss sum ${\mathcal G_a}$ is exactly $q^{1/2}.$ Moreover, the explicit form of the Gauss sum ${\mathcal G_1}$ is provided in the next lemma. 
\begin{lemma}[\cite{LN97}, Theorem 5.15]\label{ExplicitGauss}
Let $\mathbb F_q$ be a finite field with $ q= p^{\ell},$ where $p$ is an odd prime and $\ell \in {\mathbb N}.$
Then we have
$${\mathcal G_1}= \left\{\begin{array}{ll}  {(-1)}^{\ell-1} q^{\frac{1}{2}} \quad &\mbox{if} \quad p \equiv 1 \mod 4 \\
{(-1)}^{\ell-1} i^\ell q^{\frac{1}{2}} \quad &\mbox{if} \quad p\equiv 3 \mod 4.\end{array}\right. $$
\end{lemma}
Notice that $q=p^l\equiv 3\mod 4$ if and only if $p\equiv 3\mod 4$ and $l$ is an odd positive integer.

The following formula will be used in our proof of Theorem \ref{eigenvalue}. For $a\in \mathbb F_q^*$ and $b\in \mathbb F_q,$ 
\begin{equation}\label{ComSqu}  
 \sum_{s\in \mathbb F_q} \chi(as^2+bs)= \eta(a){\mathcal G_1} \chi\left(\frac{b^2}{-4a}\right).\end{equation}
This can be proved easily by completing the square and using a change of variables.

\section{Proof of Theorem \ref{eigenvalue}} \label{sec3}
In this section, we provide a proof of Theorem \ref{eigenvalue} and also for other dimensions. 
\begin{theorem}\label{eigenvalue1}
Let $\{\lambda_{m}\}_{m\in \mathbb{F}_q^k}$ be the eigenvalues of $G_{Q, k}$. 
\begin{enumerate}
\item
If $k=4n$ for some $n\in \mathbb N,$ and $q\equiv 3 \mod{4},$ then we have
 $$ \lambda_m=q^k\cdot \left\{ \begin{array}{ll} q^{-1}\delta_0(m)- q^{-\frac{k}{2}}+q^{-\frac{(k+2)}{2}}  \quad &\mbox{if}~~ m\in C_k\\
 q^{-\frac{(k+2)}{2}}   \quad &\mbox{if} ~~ m\notin C_k. \end{array}\right.$$
\item If $k=4n$ for some  $n\in \mathbb N$ and $q\equiv 1 \mod{4},$ or $k=4n+2$ for some $n\in \mathbb N$, then we have
$$ \lambda_m=q^k\cdot \left\{ \begin{array}{ll} q^{-1}\delta_0(m)+ q^{-\frac{k}{2}}-q^{-\frac{(k+2)}{2}}  \quad &\mbox{if}~~ m\in C_k\\
 -q^{-\frac{(k+2)}{2}}   \quad &\mbox{if} ~~ m\notin C_k. \end{array}\right.$$
\item
If $k\ge 3$ is odd, then we have
$$ \lambda_m= q^{k-1}\delta_{0}(m) + q^{-1} \eta(Q(m)) {\mathcal G_1^{k+1}},$$
where $\mathcal{G}_1$ is the Gauss sum defined in \eqref{GaussDef}, $\eta$ is the quadratic character of $\mathbb{F}_q^*$, and we use the convention that $\eta(0)=0.$ 
 \end{enumerate}
\end{theorem}

As we observed in the introduction, when $k=4n$ and $q\equiv 3\mod 4$,  the unique positive and non-trivial eigenvalue of this graph is much smaller than the absolute value of others. This does not hold for other dimensions or $q\equiv 1\mod 4$.

Since $G_{Q, k}$ is a Cayley graph, it is well-known in the literature that its eigenvalues can be expressed in the form $q^k\cdot \widehat{C_k}(m)$. It is sufficient to prove the following lemma. 
\begin{lemma}\label{FTForm} For any $m\in \mathbb F_q^k,$ we have
\begin{equation}\label{FFC}\widehat{C_k}(m)=q^{-1}\delta_{0}(m) +q^{-k-1} \eta(-1) {\mathcal G_1^k} \sum_{s\ne 0} \eta^k(s) \chi\left( \frac{Q(m)}{-4s}\right),\end{equation}
where $\delta_0(m)=1$ if and only if $m=(0, \ldots, 0)$. 
In particular, we have the followings:
\begin{enumerate}
\item
If $k=4n$ for some $n\in \mathbb N,$ and $q\equiv 3 \mod{4},$ then we have
 $$ \widehat{C_k}(m)=\left\{ \begin{array}{ll} q^{-1}\delta_0(m)- q^{-\frac{k}{2}}+q^{-\frac{(k+2)}{2}}  \quad &\mbox{if}~~ m\in C_k\\
 q^{-\frac{(k+2)}{2}}   \quad &\mbox{if} ~~ m\notin C_k. \end{array}\right.$$
\item If $k=4n$ for some  $n\in \mathbb N$ and $q\equiv 1 \mod{4},$ or $k=4n+2$ for some $n\in \mathbb N$, then we have
$$ \widehat{C_k}(m)=\left\{ \begin{array}{ll} q^{-1}\delta_0(m)+ q^{-\frac{k}{2}}-q^{-\frac{(k+2)}{2}}  \quad &\mbox{if}~~ m\in C_k\\
 -q^{-\frac{(k+2)}{2}}   \quad &\mbox{if} ~~ m\notin C_k. \end{array}\right.$$
\item
If $k\ge 3$ is odd, then we have
$$ \widehat{C_k}(m)= q^{-1}\delta_{0}(m) + q^{-k-1} \eta(Q(m)) {\mathcal G_1^{k+1}},$$
where we use the convention that $\eta(0)=0.$
 \end{enumerate}
 \end{lemma}
 
We note that our proof of Lemma \ref{FTForm} is quite similar to that of  \cite[Proposition 2.4]{KLP} with $Q$ is defined by $Q(x)=-x_1\cdot x_2+x_3^2+\cdots+x_k^2$.
 \begin{proof}
 By the definition and the orthogonality of $\chi,$  we have
\begin{align*}
 \widehat{C_k}(m)&= q^{-k}\sum_{x\in C_k} \chi(-x\cdot m)\\
&= q^{-1} \delta_0(m)+ q^{-k-1}\sum_{x\in {\mathbb F_q^k}}\sum_{s\neq 0} \chi\left( s(-x_1^2+x_2^2+\cdots+x_k^2)\right)~\chi(-x\cdot m)\\
&=q^{-1} \delta_0(m)+ q^{-k-1} \sum_{s\neq 0}\sum_{x_1\in \mathbb F_q} \chi(-sx_1^2-m_1x_1) \prod_{j=2}^{k} \sum_{x_j\in {\mathbb F_q}} \chi(s x_j^2-m_jx_j) .
\end{align*}
By the complete square formula \eqref{ComSqu}, we obtain \eqref{FFC} which states that
$$\widehat{C_k}(m)=q^{-1}\delta_{0}(m) +q^{-k-1} \eta(-1) {\mathcal G_1^k} \sum_{s\ne 0} \eta^k(s) \chi\left( \frac{Q(m)}{-4s}\right).$$
We now fall into three cases.

{\bf Case $1$:} Suppose that $k=4n$ for some $n\in \mathbb N,$ and $q\equiv 3 \mod{4}.$
Then $\eta^k\equiv 1$ and $\eta(-1)=-1.$ One can use Lemma \ref{ExplicitGauss} to see that  ${\mathcal G_1^k}=q^{k/2}$ for $k\equiv 0 \mod{4}.$ So, $\eta(-1) {\mathcal G_1^k}=-q^{k/2}$. This implies
$$\widehat{C_k}(m)=q^{-1}\delta_{0}(m) -q^{-k-1}q^{k/2} \sum_{s\ne 0}  \chi\left( \frac{Q(m)}{-4s}\right).$$
Thus, the first part of the lemma follows by the orthogonality of $\chi.$

{\bf Case $2$:} 
Assume that  $k=4n$ for some  $n\in \mathbb N$ and $q\equiv 1 \mod{4}$, or $k=4n+2$ for some $n\in \mathbb N.$  Using the same argument as in the previous case, it suffices to show that 
\[\eta(-1) {\mathcal G_1^k}=q^{k/2}.\]
We first assume that $k\equiv 0 \mod{4}$ and $q\equiv 1 \mod{4},$  then $-1$ is a square number, i.e.,  $\eta(-1)=1,$ and ${\mathcal G_1^k}=q^{k/2}$ by Lemma \ref{ExplicitGauss}. Hence, we get  $\eta(-1){\mathcal G_1^k}= q^{\frac{k}{2}},$ as required. 

If $k\equiv 2 \mod{4}$, then $k-2\equiv 0\mod{4}$,  so 
${\mathcal G_1^{k-2}}=q^{(k-2)/2}.$ One can use Lemma \ref{ExplicitGauss} again to obtain that  ${\mathcal G_1^2}=\eta(-1)q.$  Hence, ${\mathcal G_1^k}= \eta(-1) q {\mathcal G_1^{k-2}}=\eta(-1) q^{k/2}.$ In other words, we have proved that  $\eta(-1){\mathcal G_1^k} = q^{k/2}$.

{\bf Case $3$:} Suppose that $k\ge 3$ is an odd integer. Since $\eta^k=\eta$, it follows 
$$\widehat{C_k}(m)=q^{-1}\delta_{0}(m) +q^{-k-1} \eta(-1) {\mathcal G_1^k} \sum_{s\ne 0} \eta(s) \chi\left( \frac{Q(m)}{-4s}\right).$$
If $Q(m)=0$, then we are done by  the orthogonality of $\eta$. On the other hand, if $Q(m)\ne 0,$ then the above summation over $s\ne 0$ is the same as the quantity $\eta(-1) \eta(Q(m)) {\mathcal G_1},$ which follows by  a change of variables by letting $t=\frac{Q(m)}{-4s}.$  This completes the proof of the third part.
 \end{proof} 
Since eigenvalues of $G_{Q, k}$ are $q^k\cdot \widehat{C_k}(m)$ with $m\in \mathbb{F}_q^k$, Theorem \ref{eigenvalue1} follows directly from Lemma \ref{FTForm}.

\section{Proof of Theorem \ref{eigenvalue198}} \label{sec4}
Eigenvalues of $G_{||\cdot ||, k}$ are of the form $q^k\cdot \widehat{S_0^{k-1}}(m)$. Thus, Theorem \ref{eigenvalue198} follows directly from \cite[Lemma 2.2]{iosevichcanada}. For the reader convenience, we recall it here. 
\begin{lemma}
Assume $k=4n+2$ for some $n\in \mathbb{N}$, and $q\equiv 3\mod 4$, then we have 
\[\widehat{S_0^{k-1}}(m)=q^{-1}\delta_0(m)-q^{-\frac{k+2}{2}}\sum_{r\ne 0}\chi(r||m||).\]
\end{lemma}
This lemma was deduced from the following general statement, which can be found in \cite[Lemma 2.3]{iosevichcanada} or \cite[Lemma 4]{IKOK}.
\begin{lemma}\label{422}
For $m\in \mathbb{F}_q^k$, we have 
\[\widehat{S_0^{k-1}}=q^{-1}\delta_0(m)+q^{-k-1}\eta^{k}(-1)\mathcal{G}_1^k\sum_{r\ne 0}\eta^k(r)\chi\left(\frac{||m||}{4r}\right).\]
\end{lemma}

When $k=4n+2$ and $q\equiv 1 \mod 4$, or $k=4n$, one can compute from Lemma \ref{422} that 
 $$ \lambda_m=q^k\cdot \left\{ \begin{array}{ll} q^{-1}\delta_0(m)+q^{-\frac{k}{2}} -q^{-\frac{k+2}{2}} \quad &\mbox{if}~~ ||m||=0\\
 -q^{-\frac{k+2}{2}}   \quad &\mbox{if} ~~ ||m||\ne 0. \end{array}\right.$$

When $k$ is odd,  we have
$$ \lambda_m= q^{k-1} \delta_0(m) + q^{-1} \eta(-||m||) \mathcal{G}_1^{k+1}.$$


\section{Proofs of Theorems \ref{new-incidence} and \ref{new-incidence11}} \label{sec5}
To prove the incidence bounds, we make use of the following lemma, which can be easily proved by following the proof of the Expander mixing lemma \cite[Theorem 2.11]{alon3} and using the fact that the only positive non-trivial eigenvalue of $G_{Q, k}$ is $q^{\frac{k-2}{2}}$ when $k\equiv 0\mod 4$ and $q\equiv 3\mod 4$. The interested reader can also find a similar proof in  \cite[Lemma 2.6]{KPV}. 
\begin{lemma}\label{expander}
Suppose that $k\equiv 0\mod 4$ and $q\equiv 3\mod 4$. Let $Q(x)=-x_1^2+\sum_{i=2}^kx_i^2$. Let $W$ be a vertex set in $G_{Q, k}$ and $e(W, W)$ be the number of edges in $W$, then we have 
\[e(W, W)\le \frac{|W|^2}{q}+q^{\frac{k-2}{2}}|W|.\]
\end{lemma}
We are ready to prove Theorem \ref{new-incidence}. 
\begin{proof}[Proof of Theorem \ref{new-incidence}]
Set $k=d+1$. We identify each point $p=(p_1, \ldots, p_d)$ in $P$ with $(0, p)\in \mathbb{F}_q^k$ and each sphere $s$ centered at $a\in \mathbb{F}_q^d$ of square radius $r^2$ with $(r, a)\in \mathbb{F}_q^d$. It is clear that there is an incidence between the point $p$ and the sphere $s$ if $(p_1-a_1)^2+\cdots+(p_d-a_d)^2=(r-0)^2$. This means that $(0,p)-(r, a)\in C_k$, i.e. an edge between $(0, p)$ and $(r, a)$ in $G_{Q, k}$. Let $P'$ and $S'$ be the sets of corresponding points in $\mathbb{F}_q^k$. We have $I(P, S)=e(P', S')$. Set $W=P'\cup S'$. It is clear that $e(P', S')\le e(W, W)$. Thus, the theorem follows from Lemma \ref{expander} and theorem's assumptions. 
\end{proof}
\begin{proof}[Proof of Theorem \ref{new-incidence11}]
Set $k=d+1$. The argument is the same as what we did for Theorem \ref{new-incidence}, except that we use the graph $G_{||\cdot||, k}$ in place of $G_{Q, k}$. 
\end{proof}
\section{Proof of Theorem \ref{sum-product}} \label{sec6}
To prove Theorem \ref{sum-product}, we need to deal with two cases $d\equiv 3\mod 4$ and $d\equiv 1\mod 4$. However, the proofs for these two situations are almost identical, so we only present an argument for $d\equiv 3\mod 4$. In particular, we will show that 
\begin{theorem}\label{sum-product1111}
Let $A$ be a set in $\mathbb{F}_q$ with $q\equiv 3\mod 4$ and $|A|\gg q^{1/2}$. For $d\equiv 3\mod 4$, we  have at least one of two following statements:
\begin{enumerate}
\item $|A+A|\ge \min \left\lbrace q^{\frac{d+1}{2d}}, |A|^{\frac{d+1}{d}}\right\rbrace$.
\item $|dA^2|\gg \frac{|A|^d}{q^{\frac{d-1}{2}}}$. 
\end{enumerate}
\end{theorem}

For a triple $(x, y, z)\in A\times A\times A$, we say that $(x, y, z)$ is a \textit{square triple} if $x^2+y^2+z^2$ is a square, otherwise, we say it is a \textit{non-square triple}. 

For $A\subset \mathbb{F}_q$, define $A^2:=\{x^2\colon x\in A\}$. A tuple $(x_1, \ldots, x_d)\in A^d$ is called \textit{square-sum-type} if $x_1^2+\cdots+x_d^2$ is a square in $\mathbb{F}_q$. The next lemma shows that most $d$-tuples in $A^d$ are of square-sum-type.
\begin{lemma}\label{square}
Any set $A\subset \mathbb{F}_q$ with $|A|\gg q^{1/2}$ has at least $\gg |A|^d$ square-sum-type tuples. 
\end{lemma}
\begin{proof}
Let $SQ(\mathbb{F}_q)$ be the set of non-zero square elements in $\mathbb{F}_q$, and $B$ be the multi-set defined by $B:=\{x_1^2+x_2^2+\cdots+x_{d-1}^2\colon x_i\in A\}$. We write $\overline{B}$ for the set of distinct elements in $B$ and $\sum_{b\in \overline{B}}m(b)^2$, where $m(b)$ is the multiplicity of $b$, is the number of tuples \[(x_1, \ldots, x_{d-1}, y_1, \ldots, y_{d-1})\in A^{2(d-1)}\] such that $x_1^2+\cdots+x_{d-1}^2=y_1^2+\cdots+y_{d-1}^2$. We denote the number of these tuples by $E^+((d-1)A^2, (d-1)A^2)$. One can check that $E^+((d-1)A^2, (d-1)A^2)\ll |A|^{2d-3}$.  
We now consider the following equation 
\begin{equation}\label{eq0918}xy=a+b,\end{equation}
where $x, y\in SQ(\mathbb{F}_q), a\in A^2, b\in B$. Let $M$ be the number of solutions of this equation.  Since $|SQ(\mathbb{F}_q)|=\frac{q-1}{2}$, it follows from Lemma $2.1$ in \cite{VA} that 
\[\left\vert M-\frac{|A^2||B|(q-1)^2}{4q}\right\vert\le q^{1/2}\left(\frac{(q-1)E^+((d-1)A^2, (d-1)A^2)}{2}\right)^{1/2}\left(\frac{(q-1)|A^2|}{2}\right)^{1/2}.\]

Thus, $M\gg \frac{|A|^d(q-1)^2}{q}$ if $E^+((d-1)A^2, (d-1)A^2)\ll \frac{|A|^{2d-1}}{q}$, which can be satisfied under the condition $|A|\gg q^{1/2}$, since $E^+((d-1)A^2, (d-1)A^2)\ll |A|^{2d-3}$. 

Observe that for $a\in A^2$ and $b\in B$, if $a+b$ is a square, then it contributes $(q-1)/2$ solutions to the equation (\ref{eq0918}). Hence, the number of square-sum-type tuples $(a_1, \ldots, a_d)\in A^d$ is at least $\frac{2M}{q-1}$. This completes the proof of the lemma. 
\end{proof}

We are now ready to prove Theorem \ref{sum-product1111}
\begin{proof}[Proof of Theorem \ref{sum-product1111}]
If $|A+A|\ge |A|^{(d+1)/d}$ or $|A+A|\ge q^{\frac{d+1}{2d}}$, then we are done. Without loss of generality, we assume that $|A+A|< |A|^{\frac{d+1}{d}}$ and $|A+A|< q^{\frac{d+1}{2d}}$.

We consider the following equation 
\begin{equation}\label{eq1} (x_1-y_1)^2+(x_2-y_2)^2+\cdots+(x_d-y_d)^2=t,\end{equation}
where $x_1, x_2, \ldots, x_d\in A+A, y_1, y_2, \ldots, y_d\in A, t\in dA^2\cap SQ(\mathbb{F}_q) =(A^2+\cdots+A^2)\cap SQ(\mathbb{F}_q)$. 

Let $M$ be the number of solutions of this equation. 

By Lemma \ref{square}, the number of square-sum-type tuples in $A^d$ is at least $\gg |A|^d$. For each of those tuples, denoted by $(a_1, \ldots, a_d)$, it will contribute $\gg |A|^d$ solutions to the number of solutions of the equation (\ref{eq1}). Indeed, tuples with $(x_1,\ldots, x_d, y_1, \ldots, y_d)=(y_1+a_1, \ldots, y_d+a_d, y_1, \ldots, y_d)$ with $y_i\in A$ satisfy the equation (\ref{eq1}).  Therefore, $M\gg |A|^{2d}$. 

Define $P:=(A+A)\times (A+A)\times \cdots\times (A+A)\subset \mathbb{F}_q^d$ and $S$ be the set of spheres centered at points in $A^d$ of square radii in $dA^2$. We have $|P|=|A+A|^d$ and $|S|=|A|^d|dA^2|$. 

To apply Theorem \ref{new-incidence} effectively, one has to have the condition $|P|\sim |S|$. To this end, we partition the radius set into $m$ subsets of size $\frac{|A+A|^d}{|A|^d}$, where $m=\frac{|dA^2||A|^d}{|A+A|^d}>1$ since otherwise $|A+A|\ge |A|^{(d+1)/d}$. We denote those radius sets by $R_1, \ldots, R_m$. For $1\le i\le m$, let $S_i$ be the set of spheres centered at points in $A^d$ of square radii in $R_i$. Notice that, for each $i$, $S_i$ can be an empty set if there is no square element in $R_i$, but what we only need is an upper bound of $S_i$ which is $|A+A|^d$. 

One can check that $M$ is bounded by $\sum_{i=1}^mI(P, S_i)$. For each $i$, applying Theorem \ref{new-incidence} gives us
\[I(P, S_i)\le \frac{|A+A|^{2d}}{q}+q^{\frac{d-1}{2}}|A+A|^d\ll q^{\frac{d-1}{2}}|A+A|^d,\]
since $|A+A|\le q^{\frac{d+1}{2d}}$. 
Taking the sum over all $i$, we achieve
\[M=\sum_{i=1}^mI(P, S_i)\le q^{\frac{d-1}{2}}|A+A|^d\cdot \frac{|dA^2||A|^d}{|A+A|^d}=q^{\frac{d-1}{2}}|dA^2||A|^d.\]
Using the fact that $M\gg |A|^{2d}$ leads to 
\[|dA^2|\gg \frac{|A|^d}{q^{\frac{d-1}{2}}}.\] 
This completes the proof of the theorem.
\end{proof}
\section{Open questions}
Theorems \ref{new-incidence} and \ref{new-incidence11} give some answers for Question \ref{qs}, but we do not know whether or not Theorem \ref{motbig} can be improved for the following cases:
\begin{enumerate}
\item (square radii): $d\equiv 3\mod 4$ and $q\equiv 1\mod 4$. 
\item (square radii): $d\equiv 1\mod 4$.
\item (square radii): $d\equiv 2\mod 4$ and $q\equiv 1\mod 4$. 

\vspace{0.5cm}
\item (non-square radii): $d\equiv 1\mod 4$ and $q\equiv 1\mod 4$. 
\item (non-square radii): $d\equiv 3\mod 4$.
\item (non-square radii): $d$ is even.
\end{enumerate}
We hope to address these cases in a subsequent paper.

\section*{Acknowledgements}
Doowon Koh was supported by the National Research Foundation of Korea (NRF) grant funded by the Korea government (MIST) (No. NRF-2018R1D1A1B07044469). T. Pham was supported by Swiss National Science Foundation grant P400P2-183916.

\end{document}